\documentclass[reqno,10pt]{amsart}
\usepackage{amssymb,bm,mathrsfs,tikz,setspace,hyperref}
\usetikzlibrary{matrix}
\usepackage[vcentermath]{youngtab}
\usepackage{listings}
\lstdefinelanguage{Sage}[]{Python}
{morekeywords={False,sage,True},sensitive=true}
\definecolor{dblackcolor}{rgb}{0.0,0.0,0.0}
\definecolor{dbluecolor}{rgb}{0.01,0.02,0.7}
\definecolor{dgreencolor}{rgb}{0.2,0.4,0.0}
\definecolor{dgraycolor}{rgb}{0.30,0.3,0.30}

\newcommand{\BB}{\mathcal{B}}
\newcommand{\bd}{{\boldsymbol{\delta}}}
\newcommand{\CB}{\mathscr{B}}
\newcommand{\cc}{\bm{c}}
\newcommand{\CC}{\mybb{C}}
\newcommand{\cl}{\mathrm{cl}}
\newcommand{\im}{\mathrm{Im}}
\newcommand{\ii}{{\mathbf i}}
\newcommand{\KK}{\mathcal{K}}
\newcommand{\GG}{\mathcal{G}}

\newcommand{\M}{\mathscr{M}}
\newcommand{\mult}{\operatorname{mult}}
\newcommand{\mybb}[1]{\mathbf{#1}}
\newcommand{\N}{\mathscr{N}}

\newcommand{\pp}{\mathfrak{p}}
\newcommand{\PP}{\mathcal{P}}

\newcommand{\qq}{\mathfrak{q}}

\newcommand{\re}{\mathrm{Re}}

\newcommand{\seg}{\operatorname{seg}}
\newcommand{\wt}{{\rm wt}}
\newcommand{\YY}{\mathcal{Y}}
\newcommand{\zz}{{\bm z}}
\newcommand{\ZZ}{\mybb{Z}}

\newcommand{\GYW}[1]{
 \foreach \x [count=\s from 0] in {#1}{
   \foreach \y [count=\t from 0] in \x {
     \node[font=\scriptsize] at (-\t,\s) {$\y$};
     \draw (-\t+.5,\s+.5) to (-\t-.5,\s+.5);
     \draw (-\t+.5,\s-.5) to (-\t-.5,\s-.5);
     \draw (-\t-.5,\s-.5) to (-\t-.5,\s+.5);
   }
 \draw[-,thick] (.5,\s+1) to (.5,-.5) to (-\t-1,-.5);
 }
}

\newcommand{\wall}[1]{
	\begin{tikzpicture}[baseline=20,scale=.45]
		\GYW{#1}
	\end{tikzpicture}
}

\theoremstyle{plain}
\newtheorem{thm}{Theorem}[section]
\newtheorem{lemma}[thm]{Lemma}
\newtheorem{prop}[thm]{Proposition}
\newtheorem{cor}[thm]{Corollary}
\theoremstyle{definition}
\newtheorem{dfn}[thm]{Definition}
\newtheorem{ex}[thm]{Example}
\newtheorem{remark}[thm]{Remark}
\newtheorem{alg}[thm]{Algorithm}
\numberwithin{equation}{section}
\numberwithin{figure}{section}
\numberwithin{table}{section}

\title[Combinatorics of the Gindikin-Karpelevich Formula in affine $A$ type]{A combinatorial description of the affine Gindikin-Karpelevich formula of type $A_n^{(1)}$}
\author{Seok-Jin Kang$^{1}$} 
\thanks{$^1$ This work was supported by NRF Grant \# 2011-0017937 and NRF Grant \# 2011-0027952.} 
\address{Department of Mathematical Sciences and Research Institute of Mathematics \\ Seoul National University \\ Gwanak-ro 599\\ Gwanak-gu\\ Seoul 151-747 \\ Korea} 
\email{sjkang@math.snu.ac.kr} 
\author{Kyu-Hwan Lee}
\address{Department of Mathematics \\ University of Connecticut \\ 196 Auditorium Road, Unit 3009 \\ Storrs, CT  06269-3009 \\ USA}
\email{khlee@math.uconn.edu}
\author{Hansol Ryu$^{2}$} 
\address{Department of Mathematical Sciences \\ Seoul National University \\ Gwanak-ro 599\\ Gwanak-gu\\ Seoul 151-747 \\ Korea} 
\thanks{$^{2}$ This work was supported by BK21 Mathematical Sciences Division and NRF Grant \# 2011-0027952.} 
\email{sol8586@snu.ac.kr} 
\author{Ben Salisbury$^{3}$}
\address{Department of Mathematics \\ Central Michigan University \\ Mount Pleasant, MI 48859 \\ USA}
\thanks{$^{3}$ This work was partially supported by NSF grant DMS 0847586.}
\email{ben.salisbury@cmich.edu}
\date{\today}

\begin{document}

\begin{abstract}
The classical Gindikin-Karpelevich formula
appears in Langlands' calculation of the constant terms of Eisenstein series on reductive groups and in Macdonald's work on $p$-adic groups and affine Hecke algebras. The formula has been generalized in the work of Garland to the affine Kac-Moody case, and the affine case has been geometrically constructed in a recent paper of Braverman, Finkelberg, and Kazhdan. On the other hand, there have been efforts to write the formula as a sum over Kashiwara's crystal basis or Lusztig's canonical basis, initiated by Brubaker, Bump, and Friedberg. In this paper, we write the affine Gindikin-Karpelevich formula as a sum over the crystal of generalized Young walls when the underlying Kac-Moody algebra is of affine type $A_n^{(1)}$. The coefficients of the terms in the sum are determined explicitly by the combinatorial data from Young walls.

\end{abstract}
\maketitle

\setcounter{section}{-1}
\section{Introduction}

The classical Gindikin-Karpelevich formula originated from a certain integration on real reductive groups \cite{GK:62}. When Langlands calculated the constant terms of Eisenstein series on reductive groups \cite{Lan:71}, he considered a $p$-adic analogue of the integration and called the resulting formula the {\it Gindikin-Karpelevich formula}. In the case of $\operatorname{GL}_{n+1}$, the formula can be described as follows:
let $F$ be a $p$-adic field with residue field of $q$ elements and let $N_-$ be the maximal unipotent subgroup of $\operatorname{GL}_{n+1}(F)$ with maximal torus $T$. Let $f^\circ$ denote the standard spherical vector corresponding to an unramified character $\chi$ of $T$, let $T(\CC)$ be the maximal torus in the $L$-group $\operatorname{GL}_{n+1}(\CC)$ of $\operatorname{GL}_{n+1}(F)$, and let $\zz \in T(\CC)$ be the element corresponding to $\chi$ via the Satake isomorphism.
Then the Gindikin-Karpelevich formula is given by
\begin{equation}\label{eq:GK}
\int_{N_-(F)} f^\circ({\bm n}) \,\mathrm{d}{\bm n} = \prod_{\alpha \in \Delta^+}
\frac{1-q^{-1}\zz^\alpha}{1-\zz^\alpha},
\end{equation}
where $\Delta^+$ is the set of positive roots of $\operatorname{GL}_{n+1}(\CC)$.
The formula appears in Macdonald's study on $p$-adic groups and affine Hecke algebras as well \cite{Mac:71}, and the product side of \eqref{eq:GK} is also known as Macdonald's $c$-function.

In the paper \cite{Ga:04}, Garland generalized Langlands' calculation to affine Kac-Moody groups and obtained an affine Gindikin-Karpelevich formula as a product over $\Delta^+ \cap w^{-1}(\Delta^{-})$ for each $w \in W$, where $\Delta^+$ (resp.\ $\Delta^-$) is the set of positive (resp.\ negative) roots of the corresponding affine Kac-Moody algebra and $W$ is the Weyl group.
In a recent paper of Braverman, Finkelberg, and Kazhdan \cite{BFK}, the authors interpreted the classical Gindikin-Karpelevich formula in a geometric way, and generalized the formula to affine Kac-Moody groups and obtained another version of affine Gindikin-Karpelevich formula, which has an additional ``correction factor" in the product side.

On the other hand, in the works of Brubaker, Bump and Friedberg \cite{BBF:11}, Bump and
Nakasuji \cite{BN:10}, and McNamara \cite{McN:11}, the product side of the classical Gindikin-Karpelevich formula in type $A_n$ was written as a sum over the crystal $\BB(\infty)$. (For the definition of a crystal, see \cite{HK:02, Kash:02}.) More precisely, they proved
\[
\prod_{\alpha\in \Delta^+} \frac{1-q^{-1}\zz^\alpha}{1-\zz^\alpha} = \sum_{b\in
\BB(\infty)} G_\ii^{(e)}(b)q^{\langle \wt(b),\rho\rangle} \zz^{-\wt(b)},
\]
where $\rho$ is the half-sum of the positive roots, $\wt(b)$ is the weight of $b$, and the coefficients $G_\ii^{(e)}(b)$ are defined using so-called BZL paths or Kashiwara's parametrization. As shown in \cite{KL:11} by H. Kim and K.-H. Lee, one can also use Lusztig's parametrization of canonical bases (\cite{Luszt:90, Luszt:91}) and the product can be written as
\begin{equation} \label{eqn-kl}
\prod_{\alpha\in\Delta^+} \frac{1-q^{-1}\zz^\alpha}{1-\zz^\alpha} = \sum_{b\in \BB(\infty)}
(1-q^{-1})^{\N(\phi_\ii(b))}\zz^{-\wt(b)},
\end{equation} where $\N(\phi_\ii(b))$ is the number of nonzero entries in Lusztig's parametrization $\phi_\ii(b)$. The equation \eqref{eqn-kl} was proved for all finite roots systems $\Delta$, and was generalized in a subsequent paper \cite{KL:11b} to the affine Kac-Moody case using the results of Beck, Chari, and Pressley \cite{BCP} and Beck and Nakajima \cite{BeckNa}.

The use of crystals connects the Gindikin-Karpelevich formula to combinatorial representation theory, since much work has been done on realizations of crystals through various combinatorial objects (e.g., \cite{Kam:10,Kang:03,KN:94,KS:97,Lit:95}).
 Indeed, for type $A_n$, K.-H. Lee and Salisbury \cite{LS-A} expressed the right side of  \eqref{eqn-kl} as a sum over marginally large Young tableaux using J. Hong and H.\ Lee's \cite{HL:08}  description of $\BB(\infty)$ and the coefficients were determined by a simple statistic $\seg(b)$ of the tableau $b$. Furthermore, the meaning of $\seg(b)$ was studied in the frameworks of Kamnitzer's MV polytope model \cite{Kam:10} and Kashiwara-Saito's geometric realization \cite{KS:97} of the crystal $\BB(\infty)$.

The goal of this paper is to extend this approach to affine type $A_n^{(1)}$ through generalized Young walls. The notion of a Young wall was first introduced by Kang \cite{Kang:03} in his extensive study of affine crystals. In the case of $\BB(\infty)$ in type $A_n^{(1)}$, J.-A.\ Kim and D.-U.\ Shin \cite{KS:10} considered
 a set of {\em generalized} Young walls to obtain a realization of $\BB(\infty)$, while H. Lee \cite{Lee:07} established a different realization. In this paper, we will adopt Kim and Shin's realization and  prove (Theorem \ref{main})
\[
\prod_{\alpha \in \Delta^+} \left( \frac{1-q^{-1}\zz^\alpha}{1-\zz^\alpha} \right)^{\mult(\alpha)} = \sum_{Y\in\YY(\infty)} (1-q^{-1})^{\N(Y)}\zz^{-\wt(Y)},
\]
where $\YY(\infty)$ is the set of reduced proper generalized Young walls and $\N(Y)$ is a certain statistic on $Y\in \YY(\infty)$.

    In type $A_n^{(1)}$, the correction factor in the formula of Braverman, Finkelberg, and Kazhdan, mentioned above is given by
\begin{equation}\label{eqn-correction}
\prod_{i=1}^n \prod_{j=1}^\infty \frac{1-q^{-i}\zz^{j\delta}}{1-q^{-(i+1)}\zz^{j\delta}}  ,
\end{equation} 
 where $\delta$ is the minimal positive imaginary root.
  In the last section we will write this correction factor as a sum over a subset of reduced proper generalized Young walls (Proposition \ref{correction}), obtain an expansion of the whole product as a sum over pairs of reduced proper generalized Young walls (Corollary \ref{cor-end}), and derive a combinatorial formula for the number of points in the intersection $T^{-\gamma} \cap S^0$ of certain orbits $T^{-\gamma}$ and $S^0$ in the (double) affine Grassmannian (Corollary \ref{cor-nup}).

\subsection*{Acknowledgements}
The authors are grateful to A. Braverman for helpful comments.

\section{General definitions}\label{sec:bases}

Let $I = \{0,1,\dots,n\}$ be an index set and let $(A,\Pi,\Pi^\vee,P,P^\vee)$ be a Cartan datum of type $A_n^{(1)}$; i.e.,
\begin{itemize}
\item $A = (a_{ij})_{i,j\in I}$ is a generalized Cartan matrix of type $A_n^{(1)}$,
\item $\Pi = \{\alpha_i:i\in I\}$ is the set of simple roots,
\item $\Pi^\vee = \{h_i:i\in I\}$ is the set of simple coroots,
\item $P^\vee = \ZZ h_1 \oplus\cdots\oplus \ZZ h_n \oplus \ZZ d$ is the dual weight lattice,
\item $\mathfrak{h} = \CC\otimes_\ZZ P^\vee$ is the Cartan subalgebra,
\item and $P = \{ \lambda\in\mathfrak{h}^* : \lambda(P^\vee) \subset \ZZ \}$ is the weight lattice.
\end{itemize}
In addition to the above data, we have a bilinear pairing $\langle\ ,\ \rangle\colon P^\vee \times P \longrightarrow \ZZ$ defined by $\langle h_i,\alpha_j \rangle = a_{ij}$ and $\langle d,\alpha_j\rangle = \delta_{0,j}$.

Let $\mathfrak{g}$ be the affine Kac-Moody algebra associated with this Cartan datum, and denote by $U_v(\mathfrak{g})$ the quantized universal enveloping algebra of $\mathfrak{g}$.  We denote the generators of $U_v(\mathfrak{g})$ by $e_i$, $f_i$ ($i\in I$), and $v^h$ ($h\in P^\vee)$.  The subalgebra of $U_v(\mathfrak{g})$ generated by $f_i$ ($i\in I$) will be denoted by $U_v^-(\mathfrak{g})$.

A $U_v(\mathfrak{g})$-crystal is a set $\BB$ together with maps
\[
\widetilde e_i,\widetilde f_i \colon \BB \longrightarrow \BB\sqcup\{0\},\qquad
\varepsilon_i,\varphi_i\colon \BB \longrightarrow \ZZ \sqcup \{-\infty\},\qquad
\wt\colon \BB \longrightarrow P
\]
satisfying certain conditions (see \cite{HK:02,Kash:95}).
The negative part $U_v^-(\mathfrak{g})$ has a crystal base (see \cite{Kash:91}) which is a $U_v(\mathfrak{g})$-crystal.  We denote this crystal by $\BB(\infty)$, and denote its highest weight element by $u_\infty$.

Finally, we will describe the set of roots $\Delta$ for $\mathfrak{g}$.  Since we are fixing $\mathfrak{g}$ to be of type $A_n^{(1)}$, we may make this explicit.  Define
\begin{align*}
\Delta_{\cl} &= \{ \pm(\alpha_i + \cdots + \alpha_{j}) : 1\le i\le j \le n\},\\
\Delta_{\cl}^+ &= \{\alpha_i + \cdots + \alpha_{j} : 1\le i\le j \le n\}
\end{align*}
to be set of classical roots and positive classical roots; {\em i.e.}, roots in the root system of $\mathfrak{g}_\cl = \mathfrak{sl}_{n+1}$.  The minimal imaginary root is $\delta = \alpha_0 + \alpha_1 + \dots + \alpha_n$.  Then
\[
\Delta_{\im} = \{ m\delta : m\in \ZZ \setminus \{0 \} \}, \ \ \ \
\Delta_{\im}^+ = \{ m\delta : m\in \ZZ_{>0} \}.
\]
We have $\Delta = \Delta_{\re} \sqcup \Delta_{\im}$ and $\Delta^+ = \Delta_{\re}^+ \sqcup \Delta_{\im}^+$, where
\begin{align*}
\Delta_{\re} &= \{ \alpha + m\delta : \alpha\in\Delta_\cl,\ m\in \ZZ \} \\
\Delta_{\re}^+ &= \{ \alpha +m\delta : \alpha\in\Delta_\cl,\ m\in \ZZ_{>0} \} \cup \Delta_\cl^+.
\end{align*}
Recall $\mult(\alpha) = 1$ for any $\alpha\in \Delta_{\re}$ and $\mult(\alpha) = n$ for any $\alpha\in \Delta_{\im}$.  For notational convenience, since $\mult(m\delta) = n$, we write
\[
\Delta_{\im}^+ = \{ m_1\delta_1,\dots,m_n\delta_n : m_1,\dots,m_n \in \ZZ_{>0} \},
\]
where each $\delta_j$ is a copy of the imaginary root $\delta$.

\section{Generalized Young walls}\label{sec:GYW}

We start by defining the board on which all generalized Young walls will be built.  Define
\begin{equation}\label{GYWboard}
\begin{tikzpicture}[baseline=75,font=\footnotesize,scale=.8]
 \draw (-6,0) to (-6,6.8);
 \draw (-5,0) to (-5,6.8);
 \draw (-4,0) to (-4,6.8);
 \draw (-3,0) to (-3,6.8);
 \draw (-2,0) to (-2,6.8);
 \draw (-1,0) to (-1,6.8);
 \draw[very thick] ( 0,0) to ( 0,6.8);
 \draw[very thick] (-7,0) to (0,0);
 \draw (-7,1) to (0,1);
 \draw (-7,2) to (0,2);
 \draw (-7,3) to (0,3);
 \draw (-7,4) to (0,4);
 \draw (-7,5) to (0,5);
 \draw (-7,6) to (0,6);
 \node at (-.5,5.5) {$1$};
 \node at (-.5,4.5) {$0$};
 \node at (-.5,3.5) {$n$};
 \node at (-.5,2.65) {$\vdots$};
 \node at (-.5,1.5) {$1$};
 \node at (-.5,0.5) {$0$};
 \node at (-1.5,5.5) {$0$};
 \node at (-1.5,4.5) {$n$};
 \node at (-1.5,3.5) {$n-1$};
 \node at (-1.5,2.65) {$\vdots$};
 \node at (-1.5,1.5) {$0$};
 \node at (-1.5,0.5) {$n$};
 \node at (-2.5,0.5) {$\cdots$};
 \node at (-2.5,1.5) {$\cdots$};
 \node at (-2.5,3.5) {$\cdots$};
 \node at (-2.5,4.5) {$\cdots$};
 \node at (-2.5,5.5) {$\cdots$};
 \node at (-3.5,5.5) {$2$};
 \node at (-3.5,4.5) {$1$};
 \node at (-3.5,3.5) {$0$};
 \node at (-3.5,2.65) {$\vdots$};
 \node at (-3.5,1.5) {$2$};
 \node at (-3.5,0.5) {$1$};
 \node at (-4.5,5.5) {$1$};
 \node at (-4.5,4.5) {$0$};
 \node at (-4.5,3.5) {$n$};
 \node at (-4.5,2.65) {$\vdots$};
 \node at (-4.5,1.5) {$1$};
 \node at (-4.5,0.5) {$0$};
 \node at (-5.5,5.5) {$0$};
 \node at (-5.5,4.5) {$n$};
 \node at (-5.5,3.5) {$n-1$};
 \node at (-5.5,2.65) {$\vdots$};
 \node at (-5.5,1.5) {$0$};
 \node at (-5.5,0.5) {$n$};
 \node at (-6.5,0.5) {$\cdots$};
 \node at (-6.5,1.5) {$\cdots$};
 \node at (-6.5,3.5) {$\cdots$};
 \node at (-6.5,4.5) {$\cdots$};
 \node at (-6.5,5.5) {$\cdots$};
 \node at (-.5,6.65) {$\vdots$};
 \node at (-1.5,6.65) {$\vdots$};
 \node at (-3.5,6.65) {$\vdots$};
 \node at (-4.5,6.65) {$\vdots$};
 \node at (-5.5,6.65) {$\vdots$};
 \node at (1.5,.5)  {$1$st row};
 \node at (1.5,1.5) {$2$nd row};
 \node at (1.5,2.65) {$\vdots$};
 \node at (1.5,3.5) {$(n+1)$st row};
 \node at (1.5,4.5) {$(n+2)$nd row};
 \node at (1.5,5.5) {$(n+3)$rd row};
 \node at (1.5,6.65) {$\vdots$};
\end{tikzpicture}
\end{equation}
In particular, the color of the $j$th site from the bottom of the $i$th column from the right in \eqref{GYWboard} is $j-i\bmod n+1$.

\begin{dfn}
A {\em generalized Young wall} is a finite collection of $i$-colored boxes $\boxed{i}$ ($i\in I$) on the board \eqref{GYWboard} satisfying the following building conditions.
\begin{enumerate}
\item The colored boxes should be located according to the colors of the sites on the board \eqref{GYWboard}.
\item The colored boxes are put in rows; that is, one stacks boxes from right to left in each row.
\end{enumerate}
\end{dfn}

For a generalized Young wall $Y$, we define the {\em weight} $\wt(Y)$ of $Y$ to be
\[
\wt(Y) = -\sum_{i\in I} m_i(Y)\alpha_i, \]
where $m_i(Y)$ is the number of $i$-colored boxes in $Y$.

\begin{dfn}
A generalized Young wall is called {\it proper} if for any $k > \ell$ and $k-\ell \equiv 0 \bmod n+1$, the number of boxes in the $k$th row from the bottom is less than or equal to that of the $\ell$th row from the bottom.
\end{dfn}

\begin{dfn}
Let $Y$ be a generalized Young wall and let $Y_k$ be the $k$th column of $Y$ from the right. Set $a_i(k)$, with $i\in I$ and $k\ge 1$, to be the number of $i$-colored boxes in the $k$th column $Y_k$.
\begin{enumerate}
\item We say $Y_k$ contains a {\it removable $\delta$} if we may remove one $i$-colored box for all $i\in I$ from $Y_k$ and still obtain a generalized Young wall.  In other words, $Y_k$ contains a removable $\delta$ if $a_{i-1}(k+1) < a_i(k)$ for all $i\in I$.
\item $Y$ is said to be {\it reduced} if no column $Y_k$ of $Y$ contains a removable $\delta$.
\end{enumerate}
\end{dfn}

Let $\YY(\infty)$ denote the set of all reduced proper generalized Young walls.
In \cite{KS:10}, Kim and Shin defined a crystal structure on $\YY(\infty)$ and proved the following theorem.
We refer the reader to \cite{KS:10} for the details.

\begin{thm}[\cite{KS:10}]
We have $\BB(\infty) \cong \YY(\infty)$ as crystals.
\end{thm}


\section{Kostant partitions}

Let
\[
\alpha_i^{(\ell)} = \alpha_i + \alpha_{i-1} + \cdots + \alpha_{i-\ell+1}, \ \ \ \ i \in I , \ 1\le \ell \le n,
\]
where the indices are understood mod $n+1$.

\begin{ex}
Let $n=2$. Then
\begin{align*}
\alpha_0^{(1)} &= \alpha_0 & \alpha_1^{(1)} &= \alpha_1 &
\alpha_2^{(1)} &= \alpha_2 \\
\alpha_0^{(2)} &= \alpha_0 + \alpha_2 & \alpha_1^{(2)} &= \alpha_1 + \alpha_0 & \alpha_2^{(2)} &= \alpha_2 + \alpha_1.
\end{align*}
\end{ex}

Let
\[
S_1 =
\left\{ (m_k\delta_k), \ (c_{i,\ell}\, \delta+ \alpha_{i}^{(\ell)}) : \substack{m_k>0,\ 1\le k \le n, \\[3pt] c_{i,\ell} \ge 0, \ i\in I,\ 1 \le \ell\le n}\right\}.
\]
We introduce the generator $\bd^{(m)}$ for $m \in \ZZ_{>0}$ and set
\[
S_2= \{ \bd^{(m)} : m \in \ZZ_{>0} \}.
\]
Let $\widetilde{\GG}$ be the free abelian group generated by $S_1 \cup S_2$.
Consider the subgroup $L$ of $\widetilde{\GG}$ generated by the elements: for $m>0$,
\begin{equation}\label{reductiondelta}
\begin{cases} \bd^{(m)} - \displaystyle
\sum_{i\in I} (k\delta+\alpha_i^{(\ell)}), & m=(n+1)k+\ell, \ 1\le \ell \le n; \\ \bd^{(m)} - \displaystyle
\bd^{(k)} - \sum_{i=1}^n (k\delta_i), & m=(n+1)k.
\end{cases}
\end{equation}
We set $\GG = \widetilde{\GG}/L$ and let $\GG^+$ be the $\ZZ_{\ge 0}$-span of $S_1 \cup S_2$ in $\GG$. The following observation will play an important role.

\begin{remark} \label{rmk-delta}
If we slightly abuse language, we may say that, in $\GG$, the element $\bd^{(m)}$ is equal to the sum of $n+1$ distinct positive ``roots" of equal length $m$ whose total weight is $m\delta$.  In particular, if $m=(n+1)k+\ell$ ($1\le \ell \le n$), then $\bd^{(m)}$ is equal to the sum of $n+1$ distinct positive real roots of equal length $m$, and if $m=(n+1)k$, then $\bd^{(m)}$ is equal to the sum of $(k\delta_1), \ldots, (k\delta_n), \bd^{(k)}$ of equal length $m$. 
\end{remark}

\begin{ex}
Let $n=2$.  Then in $\mathcal G$,
\begin{align*}
\bd^{(1)} &= (\alpha_0) + (\alpha_1) + (\alpha_2) \\
\bd^{(2)} &= (\alpha_0+\alpha_2) + (\alpha_1+\alpha_0) + (\alpha_2 + \alpha_1) \\
\bd^{(3)} &= \bd^{(1)}+ (\delta_1) + (\delta_2)= (\alpha_0) + (\alpha_1) + (\alpha_2) + (\delta_1) + (\delta_2) \\
\bd^{(4)} &= (\delta+\alpha_0) + (\delta+\alpha_1) + (\delta+\alpha_2) \\
\bd^{(5)} &= (\delta+\alpha_0+\alpha_2) + (\delta+\alpha_1+\alpha_0) + (\delta+\alpha_2+\alpha_1)\\
\bd^{(6)} & = \bd^{(2)}+ (2\delta_1) + (2\delta_2) = (\alpha_0+\alpha_2) + (\alpha_1+\alpha_0) + (\alpha_2+\alpha_1) + (2\delta_1) + (2\delta_2) \\
& \ \, \vdots \\
\bd^{(9)} &= \bd^{(3)} + (3\delta_1) + (3\delta_2) = (\alpha_0) + (\alpha_1) + (\alpha_2) + (\delta_1) + (\delta_2) + (3\delta_1) + (3\delta_2) \\
& \ \, \vdots
\end{align*}
\end{ex}

\begin{dfn}
Let $\pp \in \GG^+$, and write $\pp$ as a $\ZZ_{\ge 0}$-linear combination of elements in $S_1 \cup S_2$.
\begin{enumerate}
\item We say an expression of $\pp$ contains a {\it removable $\delta$} if it contains some parts that can be replaced by $\bd^{(k)}$ for some $k>0$.
\item We say an expression of $\pp$ is {\it reduced} if it does not contain a removable $\delta$.
\end{enumerate}
\end{dfn}

 Let $\KK(\infty)$ denote the set of reduced expressions of elements in $\GG^+$. We define the set $\KK$ of {\it Kostant partitions} to be the $\ZZ_{\ge 0}$-span of the set $S_1$ in $\GG^+$. Notice that the set $S_1$ is linearly independent.

\begin{dfn}
 For $\pp \in \KK$, we denote by $\N(\pp)$ the number of distinct parts in $\pp$.
\end{dfn}

\begin{ex}
If $\pp = 2(\alpha_0+\alpha_1)  + 5(\alpha_2+\alpha_1) +  2(\delta_1) + (\delta_2) + (\alpha_0) + 4(\alpha_1)$, then $\N(\pp)=6$.
\end{ex}

Define a {\it reduction map} $\psi\colon \KK \longrightarrow \KK(\infty)$
as follows: Given $\pp \in \KK$, write it as a $\ZZ_{\ge 0}$-linear combination of elements in $S_1$. Replace $k_1\sum_{i \in I} (\alpha_i^{(1)})$ in the expression, where $k_1$ is the largest possible, with $k_1\bd^{(1)}$. The resulting expression is denoted by $\pp^{(1)}$. Next, replace $k_2\sum_{i \in I}( \alpha_i^{(2)})$ (or $k_2( \bd^{(1)}+(\delta_1))$ if $n=1$), where $k_2$ is the largest possible, with $k_2\bd^{(2)}$. The result is denoted by $\pp^{(2)}$. Continue this process with $\bd^{(k)}$ ($k\ge 3$) using the relations in \eqref{reductiondelta}. The process stops with $\pp^{(s)}$ for some $s$. By construction, $\pp^{(s)} \in \KK(\infty)$, and we define $\psi(\pp) = \pp^{(s)}$.

Conversely, we define the {\it unfolding map} $\phi\colon \KK(\infty) \longrightarrow \KK$ by unfolding the $\bd^{(k)}$'s consecutively.  That is, given $\qq\in \KK(\infty)$,
find $\bd^{(r)}$ with the largest $r$ and replace it with the corresponding sum from \eqref{reductiondelta}. The resulting expression is denoted by $\qq^{(r)}$. Next, replace $\bd^{(r-1)}$ with the corresponding sum from \eqref{reductiondelta}. The result is denoted by $\qq^{(r-1)}$. Continue this process until we replace $\bd^{(1)}$ with $\sum_{i \in I} (\alpha_i^{(1)})$ and obtain $\qq^{(1)}$. By construction, $\qq^{(1)} \in \KK$, and we define $\phi(\qq) = \qq^{(1)}$.

It is clear from the definitions that $\psi$ and $\phi$ are inverses to each other.  Hence, we have proven the following lemma.

\begin{lemma} \label{lem-unf}
The reduction map $\psi\colon  \KK \longrightarrow \KK(\infty)$ is a bijection, whose inverse is the unfolding map $\phi\colon \KK(\infty) \longrightarrow \KK$.
\end{lemma}

For later use, we need to describe the unfolding map $\phi$ more explicitly.
\begin{lemma} \label{lem-more}
For $p \in \ZZ_{\ge 0}$ and $q \in \ZZ_{>0}$, we have
\begin{equation}\label{unfolddelta}
\phi\big(\bd^{((n+1)^{p}q)}\big) = \sum_{j=1}^{n+1} (r \delta  + \alpha_{j-1}^{(s)}) + \sum_{i=0}^{p-1} \Big(\sum_{j=1}^n \big((n+1)^{i}q \, \delta_j\big)\Big),
\end{equation}
where we write $q = (n+1)r+s$, $1 \le s \le n$. In particular, $\bd^{((n+1)^{p}q)}$ has $n+1+np$ parts.
\end{lemma}

\begin{proof}
We use induction on $p$. Assume that $p=0$. Then it follows from \eqref{reductiondelta} that
\[ \phi\big(\bd^{(q)}\big) = \sum_{j=1}^{n+1} (r \delta  + \alpha_{j-1}^{(s)}). \]
Now assume that $p \ge 1$. From \eqref{reductiondelta} and the induction hypothesis, we obtain
\begin{align*}
\phi\big(\bd^{((n+1)^{p}q)}\big)
&= \phi \big ( \bd^{((n+1)^{p-1}q)}\big)+ \sum_{j=1}^n \big((n+1)^{p-1}q \, \delta_j\big) \\
& = \sum_{j=1}^{n+1} (r \delta  + \alpha_{j-1}^{(s)}) + \sum_{i=0}^{p-2} \Big(\sum_{j=1}^n \big((n+1)^{i}q \, \delta_j\big)\Big) + \sum_{j=1}^n \big((n+1)^{p-1}q \, \delta_j\big) \\
&= \sum_{j=1}^{n+1} (r \delta  + \alpha_{j-1}^{(s)}) + \sum_{i=0}^{p-1} \Big(\sum_{j=1}^n \big((n+1)^{i}q \, \delta_j\big)\Big) .\qedhere
\end{align*}
\end{proof}

In what follows, we will establish a bijection between $\YY(\infty)$ and $\KK(\infty)$.
For $Y\in \YY(\infty)$, we define $N_{k}(Y)$ ($k\ge 1$) to be the number of boxes in the
$k$th row of $Y$.
We first define a map $\Psi\colon \YY(\infty) \longrightarrow \KK(\infty)$ by describing how the blocks in a reduced proper generalized Young wall $Y$ contribute to the parts in a reduced Kostant partition.   Let $Y\in \YY(\infty)$ and consider the correspondence, for $m \ge 0$, $1 \le j \le n$ and $1 \le \ell \le n$,
\begin{equation}\label{buildp}
\Psi\colon \begin{cases}
N_{(n+1)m+j}(Y) = (n+1)k &\mapsto (k\delta_j), \\
N_{(n+1)m+j}(Y) = (n+1)k+\ell &\mapsto (k\delta + \alpha_{j-1}^{(\ell)}), \\
N_{(n+1)(m+1)}(Y) = (n+1)k &\mapsto \bd^{(k)} ,\\
N_{(n+1)(m+1)}(Y) = (n+1)k+\ell & \mapsto (k\delta + \alpha_{n}^{(\ell)}).
\end{cases}
\end{equation}
Then $\Psi(Y)$ is defined to be the expression $\pp$ obtained by adding up all the parts prescribed by \eqref{buildp}.

\begin{lemma} \label{lem-Yp}
For any $Y\in \YY(\infty)$, we have $\Psi(Y) \in \KK(\infty)$.
\end{lemma}

\begin{proof}
Let $\pp = \Psi(Y)$.  It is clear that $\pp \in \GG^+$, so it remains to show the expression of $\pp$ is reduced. On the contrary, assume that $\pp$ contains a removable $\delta$. By Remark \ref{rmk-delta}, the expression of $\pp$ contains a sum of $n+1$ distinct positive ``roots" of equal length, and the sum corresponds through \eqref{buildp} to a collection of rows of $Y$ with equal length in non-congruent positions. Then $Y$ contains a removable $\delta$, which is a contradiction. Thus $\pp$ does not contain a removable $\delta$, so $\pp$ is reduced.
\end{proof}

\begin{ex}
Let $Y = \widetilde f_2^3 \widetilde f_0^2 \widetilde f_1^2 \widetilde f_2 \widetilde f_1 \widetilde f_0 Y_\infty$.  That is, let
\[
Y = \wall{{0,2,1,0,2},{1,0,2},{2},{},{1}}\ .
\]
Then $\Psi(Y) = (\delta+\alpha_0+\alpha_2) + (\delta_2) + (\alpha_2) + (\alpha_1)$.
\end{ex}

Now define a function $\Phi\colon \KK(\infty) \longrightarrow \YY(\infty)$ in the following way.  Let $\pp$ be a reduced Kostant partition.  To each part of the partition, we assign a row of a generalized Young wall using the following prescription. For $1 \le j \le n$ and $1 \le \ell \le n$,
\begin{equation}\label{buildY}
\Phi\colon
\begin{cases}
(k\delta_j) &\mapsto \text{$(n+1)k$ boxes in row $\equiv j$ (mod $n+1$)},\\
(k\delta + \alpha_{j-1}^{(\ell)}) &\mapsto \text{$(n+1)k+\ell$ boxes in row $\equiv j$ (mod $n+1$)},\\
\bd^{(k)} &\mapsto \text{$(n+1)k$ boxes in row $\equiv 0$ (mod $n+1$)}.
\end{cases}
\end{equation}
To construct the Young wall $\Phi(\pp)$ from this data, we arrange the rows so that the number of boxes in each row of the form $(n+1)k+j$, for a fixed $j$, is weakly decreasing as $k$ increases. Hence $\Phi(\pp)$ is proper.

\begin{lemma}
For any $\pp\in \KK(\infty)$, we have $\Phi(\pp) \in \YY(\infty)$.
\end{lemma}

\begin{proof}
We set $Y = \Phi(\pp)$.  Since $\pp$ is reduced, $\pp$ does not contain a removable $\delta$. Using a similar argument as in the proof of Lemma \ref{lem-Yp}, we see that a removable $\delta$ of $Y$ corresponds to a removable $\delta$ of $\pp$.
 Thus $Y$ does not contain a removable $\delta$, so $Y \in \YY(\infty)$.
\end{proof}

\begin{ex}
Let $\pp = (\alpha_0) + (2\delta + \alpha_1+\alpha_0) + \bd^{(3)} + (\alpha_2)$.  Then
\[
\Phi(\pp) =
\wall{{0},{1,0,2,1,0,2,1,0},{2,1,0,2,1,0,2,1,0},{},{},{2}}\ .
\]
\end{ex}

\begin{prop} \label{prop-bi}
The maps $\Psi$ and $\Phi$ are bijections which are inverses to each other. In particular, we have $\YY(\infty) \cong \KK(\infty)$ as sets.
\end{prop}

The existence of a bijection is guaranteed by the theory of Kostant partitions and crystal bases. The importance of the proposition is that we have constructed an explicit, combinatorial description of a bijection.

\begin{proof}
Assume that $Y \in \YY(\infty)$. It is enough to check that a row $j$ of $Y$ is mapped onto the same stack of boxes in a row $\equiv j$ (mod $n+1$) by $\Phi \circ \Psi$, since the rows are arranged uniquely so that the number of boxes in each row of the form $(n+1)k+j$ for a fixed $j$ is weakly decreasing as $k$ increases. It follows from \eqref{buildp} and \eqref{buildY} that  a row $j$ of $Y$ is mapped onto the same stack of boxes in a row $\equiv j$ (mod $n+1$).

Conversely, assume that $\pp \in \KK(\infty)$. It is enough to check that each part of $\pp$ is mapped onto itself through $\Psi \circ \Phi$. Using \eqref{buildp} and \eqref{buildY}, we see that it is the case.
\end{proof}

\begin{remark}
While one may define a crystal structure on $\KK(\infty)$ directly in order to show that the bijection in Proposition \ref{prop-bi} is a crystal isomorphism, the bijection given is very explicit and easily understood, so one my simply pull back the crystal structure on $\YY(\infty)$ to $\KK(\infty)$ in order to obtain a crystal isomorphism.
\end{remark}

For $1 \le j \le n+1$ and $Y\in \YY(\infty)$, define $S_j(Y)$ be the set of distinct $N_{(n+1)m+j}(Y)$'s for $m \ge 0$; {\it i.e.},
set
\[
S_j(Y) = \bigcup_{m\ge 0} \big\{ N_{(n+1)m+j}(Y) \big\}.
\]
When $j=n+1$, for each $m\ge 0$, define $(p_m, q_m) \in \ZZ_{\ge 0} \times \ZZ_{\ge 0}$ by
\[N_{(n+1)(m+1)}(Y) = (n+1)^{p_m}q_m ,\] with $q_m$ not divisible by $n+1$. If $N_{(n+1)(m+1)}(Y) =0$, then we put $(p_m, q_m)=(0,0)$. We set
\[
\mathscr Q(Y) = \Big( \bigcup_{m\ge 0} \{ (n+1)^sq_m : s=0,1,\dots,p_m-1\} \Big) \cup \{0\},
\]
and let
\[
\mathscr P(Y) = \sum_{\substack{t\ge 1\\ (n+1) \nmid t}} \max \left\{ p_m : q_m=t,\ m\ge 0  \right\}.
\]
Define
\begin{equation}\label{d-def}
\N(Y) = n\mathscr P(Y) +  \sum_{j=1}^{n+1} \#\big(S_j(Y)\setminus \mathscr Q(Y)\big).
\end{equation}

\begin{prop} \label{prop-count}
Assume that $Y \in \YY(\infty)$, and let $\pp= (\phi \circ \Psi) (Y) \in \KK$, where $\phi$ is the unfolding map defined in the proof of Lemma \ref{lem-unf}.
 Then $\N(Y)$ is equal to the number of distinct parts in the Kostant partition $\pp$; i.e., we have $\N(Y)= \N(\pp)$.
\end{prop}

Before we prove this proposition, we provide a pair of examples. In the first example, we do not have $\bd^{(k)}$ in $\Psi (Y)$, and in the second example, we have $\bd^{(k)}$ in $\Psi (Y)$. We will see how the formula for $\N(Y)$ works.

\begin{ex}\label{ex-nosplit}
Suppose that
\[
Y = \wall{{0,2,1,0,2},{1,0,2},{2},{},{1}}\ .
\]
Then $\pp=(\phi \circ \Psi)(Y)= (\delta+\alpha_0+\alpha_1) + (\delta_2) + (\alpha_2) + (\alpha_1)$, and the number of distinct parts is $4$. On the other hand,

\[
S_1(Y) = \{5, 0\},\ \ \
S_2(Y) = \{3,1,0\},\ \ \
S_3(Y) = \{1,0\}.
\]
Now setting $N_{3(m+1)}(Y) = 3^{p_m}q_m$ implies
$(p_0, q_0)=(0,1)$ and $(p_m,q_m)=(0,0)$ for $m \ge 1$.
Thus $\mathscr Q(Y) = \{0\}$ and $\mathscr P(Y) = 0$.  So
\[
\N(Y) = 1 + 2 + 1 + 2\cdot 0 = 4.
\]
\end{ex}

\begin{ex}\label{ex-split}
Suppose that
\[
Y = \wall{{0,2,1},{1,0},{2,1,0,2,1,0,2,1,0}, {},{},{2,1,0,2,1,0}}\ .
\]
Then we have
\begin{align*}
\pp &=  (\phi \circ \Psi)(Y)= \phi \left ( (\delta_1)+(\alpha_0+\alpha_1)+\bd^{(3)}+\bd^{(2)} \right ) \\
&= (\delta_1)+(\alpha_0+\alpha_1) + (\alpha_0) + (\alpha_1) +(\alpha_2) + (\delta_1) + (\delta_2) + (\alpha_0+\alpha_2) \\
& \ \ \ \ \ \ \ \ + (\alpha_1+\alpha_0) + (\alpha_2+\alpha_1) \\
&= 2(\alpha_1+\alpha_0) +(\alpha_0+\alpha_2) + (\alpha_2+\alpha_1) +  2(\delta_1) + (\delta_2) + (\alpha_0) + (\alpha_1) +(\alpha_2).
\end{align*}
Hence the number of distinct parts is $8$. On the other hand, we get
\[
S_1(Y) = \{3,0\},\ \ \
S_2(Y) = \{2,0\},\ \ \
S_3(Y) = \{9,6,0\}.
\]
From $N_{3(m+1)}(Y) = 3^{p_m}q_m$, we obtain $(p_0,q_0)=(2,1)$, $(p_1,q_1)=(1,2)$ and $(p_m,q_m)=(0,0)$ for $m \ge 2$.
Then $\mathscr Q(Y)= \{ 1, 3, 2 ,0\}$ and $\mathscr P(Y)= 2+1=3$.
So
\[
\N(Y) = 0+ 0+ 2 + 2\cdot 3 = 8.
\]

\end{ex}

\begin{proof}[Proof of Proposition \ref{prop-count}] \hfill

\underline{Step 1}: Assume that $p_m=0$ for all $m \ge 0$. Then $\Psi(Y)$ has no $\bd^{(k)}$, or equivalently, $Y$ is such that $N_{(n+1)(m+1)}(Y) \neq (n+1)k$ for any $m\ge 0$ and $k\ge 1$.  Then $(\phi \circ \Psi)(Y)= \Psi(Y)$ as $\Psi(Y)$ does not have a $\bd^{(k)}$. On the other hand, since $p_m=0$ for all $m\ge 0$, we have $\mathscr Q(Y) = \{0\}$ and $\mathscr P(Y) = 0$.  Hence
\[
\N(Y) = \sum_{j=1}^{n+1} \# \big (S_j(Y) \setminus \{ 0 \} \big).
\]
For each $1 \le j \le n+1$, define $R_j(Y)$ to be the collection of $k$th rows of $Y$ with $k \equiv j$ (mod $n+1)$. From \eqref{buildp}, we see that two nonempty rows $y_1, y_2 \in R_j(Y)$ correspond to distinct parts in $\Psi(Y)$ if and only if the lengths of $y_1$ and $y_2$ are different. Since $\# \left (S_j(Y) \setminus \{ 0 \} \right )$ is the number of distinct nonzero lengths of rows in $R_j(Y)$, it is equal to the number of distinct parts in $\Psi(Y)$ corresponding to $R_j(Y)$. Furthermore, if $j \neq j'$, then $y \in R_j(Y)$ and $y' \in R_{j'}(Y)$ correspond to distinct parts in $\Psi(Y)$. Thus $\N(Y)$ is the total number of distinct parts in $\Psi(Y)=(\phi \circ \Psi)(Y)$, as required.

\underline{Step 2}: Now assume that $p_m \ge 1$ for some $m$ and $p_{m'}=0$ for all $m' \neq m$. From the definition $N_{(n+1)(m+1)}(Y)=(n+1)^{p_m}q_m$, we see that the row $(n+1)(m+1)$ has $(n+1)^{p_m}q_m$ boxes, and the corresponding part in $\Psi(Y)$ is $\bd^{((n+1)^{p_m-1}q_m)}$.  We obtain from Lemma \ref{lem-more}
\begin{equation} \label{eqn-i} \phi\big(\bd^{((n+1)^{p_m-1}q_m)}\big) = \sum_{j=1}^{n+1} (r_m \delta  + \alpha_{j-1}^{(s_m)}) + \sum_{i=0}^{p_m-2} \Big(\sum_{j=1}^n \big((n+1)^{i}q_m \, \delta_j\big)\Big),
\end{equation}
where we write $q_m=(n+1)r_m+s_m$, $1 \le s_m \le n$. Thus $\phi \big ( \bd^{((n+1)^{p_m-1}q_m)} \big )$ has $np_m+1$ distinct parts, some of which may be the same as other parts in $\Psi(Y)$.  It follows from \eqref{buildp} that the part $(r_m \delta + \alpha_{j-1}^{(s_m)})$ corresponds to $q_m$ boxes in a row $\equiv j$ (mod $n+1$) for $1 \le j \le n+1$. Similarly, the part $((n+1)^i q_m \, \delta_j)$ corresponds to $(n+1)^{i+1} q_m$ boxes in a row $\equiv j$ (mod $n+1$) for $1 \le j \le n$ and $0 \le  i \le p_m-2$. Then the number of distinct parts in $(\phi \circ \Psi)(Y)$ is
\begin{multline}
np_m+1 + \sum_{j=1}^n \# \bigl(S_j(Y) \setminus \{ 0 ,(n+1)^i q_m \}_{0 \le i \le p_m-1 }\bigr)+ \# \big(S_{n+1}(Y) \setminus \{ 0 , q_m, (n+1)^{p_{m}}q_m \}\big) \\
= np_m + \sum_{j=1}^n \# \big(S_j(Y) \setminus \{ 0 ,(n+1)^i q_m \}_{0 \le i \le p_m-1 }\big)+ \# \big(S_{n+1}(Y) \setminus \{ 0 , q_m \}\big).  \label{eqn-ii}
 \end{multline}
Since $S_{n+1}(Y)$ does not contain $(n+1)^{i}q_m$, $1 \le i \le p_m-1$, by the assumption, the expression \eqref{eqn-ii} is equal to
\begin{align*}
np_m &+ \sum_{j=1}^n \# \big(S_j(Y) \setminus \{ 0 ,(n+1)^i q_m \}_{0 \le i \le p_m-1 }\big)\\
&\ \ \ \ \ \ \ \ \ + \# \big(S_{n+1}(Y) \setminus \{ 0 , (n+1)^i q_m \}_{0 \le i \le p_m-1 }\big) \\
&=
np_m + \sum_{j=1}^{n+1} \# \big(S_j(Y) \setminus \{ 0 ,(n+1)^i q_m \}_{0 \le i \le p_m-1 }\big) \\
&= np_m + \sum_{j=1}^{n+1} \# \big(S_j(Y) \setminus \mathscr Q(Y) \big) \\
&= \N(Y) .
\end{align*}
Thus the number of distinct parts in $(\phi \circ \Psi)(Y)$ is $\N(Y)$.

\underline{Step 3}: Next we assume $p_m = \max \{ p_{m'} : m' \ge 0 \}$ and  $q_m=q_{m'}$ for any $p_{m'} \ge 1$.
We have $\bd^{((n+1)^{p_{m'}-1}q_{m'})}$ in $\Psi(Y)$ for each $p_{m'} \ge 1$, and each $\phi(\bd^{((n+1)^{p_{m'}-1}q_{m'}) })$ yields $np_{m'}+1$ parts as in \eqref{eqn-i}. However, we can see from \eqref{eqn-i} that $\phi(\bd^{(n+1)^{p_m-1}q_m})$ with the maximal $p_m$ generates all the distinct parts including those from other $p_{m'}$, since $q_m=q_{m'}$ for all $p_{m'}\ge 1$ by the assumption. Then the number of distinct parts in $(\phi \circ \Psi)(Y)$ is given by
\begin{align*}
np_m+1 &+ \sum_{j=1}^n \# \big(S_j(Y) \setminus \{ 0 ,(n+1)^i q_m \}_{0 \le i \le p_m-1 }\big)\\
&\ \ \ \ \ \ \ \ \  + \# \big(S_{n+1}(Y) \setminus \{ 0 , q_m, (n+1)^{p_{m'}}q_m \}_{1 \le p_{m'} \le p_m} \big) \\
&= np_m + \sum_{j=1}^n \# \big(S_j(Y) \setminus \{ 0 ,(n+1)^i q_m \}_{0 \le i \le p_m-1 }\big)\\
&\ \ \ \ \ \ \ \ \  + \# \big(S_{n+1}(Y) \setminus \{ 0 , (n+1)^i q_m \}_{0 \le i \le p_m-1 }\big) \\
&=  np_m + \sum_{j=1}^{n+1} \# \big(S_j(Y) \setminus \mathscr Q(Y) \big) = \N(Y) .
\end{align*}

\underline{Step 4}:  Finally we consider the general case. We group $p_{m}$'s using the rule that $p_m$ and $p_{m'}$ are in the same group if and only if $q_m=q_{m'}$. For each of such groups, we use the result in Step 3, and see that the number of distict parts in $(\phi \circ \Psi)(Y)$ is equal to
\[
n \mathscr P(Y) + \sum_{j=1}^{n+1} \# \big( S_j(Y) \setminus \mathscr Q(Y)\big),
\]
recalling the definitions
\begin{align*}
\mathscr P(Y) &= \sum_{\substack{t\ge 1\\ (n+1) \nmid t}} \max \left\{ p_m : q_m=t,\ m\ge 0  \right\}, \\
\mathscr Q(Y) &= \Big( \bigcup_{m\ge 0} \{ (n+1)^sq_m : s=0,1,\dots,p_m-1\} \Big) \cup \{0\} .
\end{align*}
Hence the number of distinct parts in $(\phi \circ \Psi)(Y)$ is $\N(Y)$.
\end{proof}

The rule for calculating the number $\N(Y)$, for $Y\in \YY(\infty)$, may be reinterpreted using the following algorithm.  For this algorithm, we say that two rows in $Y$ are {\it distinct} if their rightmost boxes are different or if their rightmost boxes are equal but they have an unequal number of boxes.

\begin{alg}
Define a  map $\psi_\YY$ on $\YY(\infty)$ as follows.
\begin{enumerate}
\item If $Y$ has no row with rightmost box $n$ and length $\equiv 0 \bmod n+1$, then $\psi_\YY(Y) := Y$.
\item If $Y$ has at least one row with rightmost box $n$ and length $(n+1)\ell$, then replace any row with maximal such $\ell$ with $n+1$ distinct rows of length $\ell$.  Rearrange all rows (if necessary) so that it is proper.  This gives $\psi_\YY^{(\ell)}(Y)$.
\item Apply Step 2 with $\ell$ replaced by $\ell-1$ and $Y$ replaced by $\psi_\YY^{(\ell)}(Y)$.  This gives $\psi_\YY^{(\ell-1)}(Y)$.
\item Iterate this process until $\ell=1$.  Then $\psi_\YY(Y) = \psi_\YY^{(1)}(Y)$.   
\end{enumerate}
Note that $\psi_\YY(Y)$ is proper, but need not be reduced, so $\psi_\YY(Y) \notin \YY(\infty)$ in general.   Then $\N(Y)$ is the number of distinct rows in $\psi_\YY(Y)$.
\end{alg}

\begin{ex}\label{ex-splitalg}
Let $n=2$ and let $Y$ be as in Example \ref{ex-split}.  Then
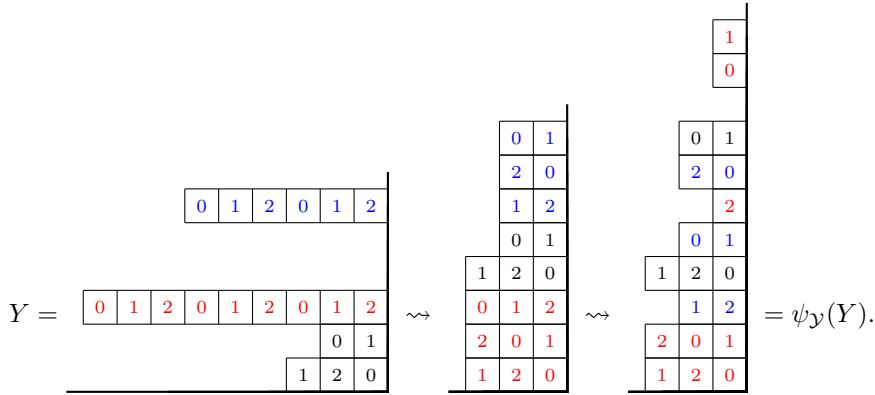
\[
Y = \wall{{0,2,1},{1,0},{\color{red}2,\color{red}1,\color{red}0,\color{red}2,\color{red}1,\color{red}0,\color{red}2,\color{red}1,\color{red}0}, {},{},{\color{blue}2,\color{blue}1,\color{blue}0,\color{blue}2,\color{blue}1,\color{blue}0}}
\  \leadsto \ 
\wall{{\color{red}0,\color{red}2,\color{red}1},{\color{red}1,\color{red}0,\color{red}2},{\color{red}2,\color{red}1,\color{red}0},{0,2,1},{1,0},{\color{blue}2,\color{blue}1},{\color{blue}0,\color{blue}2},{\color{blue}1,\color{blue}0}}
\  \leadsto \ 
\wall{{\color{red}0,\color{red}2,\color{red}1},{\color{red}1,\color{red}0,\color{red}2},{\color{blue}2,\color{blue}1},{0,2,1},{\color{blue}1,\color{blue}0},{\color{red}2},{\color{blue}0,\color{blue}2},{1,0},{},{\color{red}0},{\color{red}1}}\ = \psi_\YY(Y).
\]
Counting the number of distinct rows gives $8 = \N(Y)$.
\end{ex}

Let $W$ be the Weyl group of $\mathfrak g$ and $s_i$ ($i \in I$) be the simple reflections. We fix $\bm h =( \dots , i_{-1}, i_0, i_1, \dots)$ as in Section 3.1 in \cite{BeckNa}.
 Then for any integers $m < k$, the product $s_{i_m} s_{i_{m+1}} \cdots s_{i_k} \in W$ is a reduced expression, so is the product  $s_{i_k} s_{i_{k-1}} \cdots s_{i_m} \in W$. We set
 \begin{equation} \label{beta} \beta_k = \begin{cases} s_{i_0}s_{i_{-1}} \cdots s_{i_{k+1}} ( \alpha_{i_k}) & \text{ if } k \le 0 , \\ s_{i_1}s_{i_{2}} \cdots s_{i_{k-1}} ( \alpha_{i_k}) & \text{ if } k > 0 . \end{cases} \end{equation}

Let $T_i=T''_{i, 1}$ be the automorphism of $U_v(\mathfrak{g})$ as in Section 37.1.3. of \cite{Luszt:93}, and let
\[
\cc_+= (c_0, c_{-1}, c_{-2} , \dots ) \in \ZZ_{\ge 0}^{\ZZ_{\le 0}}
\quad \text{ and } \quad
\cc_- = (c_1, c_2, \dots ) \in \ZZ_{\ge 0}^{\ZZ_{> 0} }
\]
be functions (or sequences) that are almost everywhere zero. We denote by $\mathscr C_> $ (resp.\ by $\mathscr C_<$) the set of such functions $\cc_+$ (resp.\ $\cc_-$). For an element $\cc_+=(c_0, c_{-1},  \dots ) \in \mathscr C_>$ (resp.\ $\cc_-=(c_1, c_2, \dots ) \in \mathscr C_>$), we define
\[
E_{\cc_+} = E_{i_0}^{(c_0 )} T^{-1}_{i_0} \bigl ( E_{i_{-1}}^{( c_{-1} )} \bigr )  T^{-1}_{i_0}  T^{-1}_{i_{-1}}  \bigl ( E_{i_{-2} }^{(  c_{-2}  )} \bigr ) \cdots
\]
and
\[
E_{\cc_-} = \cdots T_{i_1}  T_{i_{2}}  \bigl ( E_{i_{3} }^{(  c_{3}  )} \bigr) T_{i_1} \bigl ( E_{i_{2}}^{( c_{2} )} \bigr )   E_{i_1}^{(c_1 )}  .
\]
We also define $\N(\cc_+)$ (resp.\ $\N(\cc_-)$) to be the number of nonzero $c_i$'s in $\cc_+$ (resp.\ $\cc_-$).

Let $\cc_0 = (\rho^{(1)}, \rho^{(2)}, \dots , \rho^{(n)} )$ be a multi-partition with $n$ components; {\it i.e.}, each component $\rho^{(i)}$ is a partition. We denote by $\mathcal P(n)$ the set of all multi-partitions with $n$ components.
Let $S_{\cc_0}$ be defined as in \cite[p.~352]{BeckNa} for $\cc_0 \in \mathcal P(n)$.
For a partition $\bm p = (1^{m_1}2^{m_2} \cdots r^{m_r} \cdots)$, we define
\begin{equation} \label{eqn-q}
\N(\bm p) = \#\{ r : m_r \neq 0 \}
\quad \text { and } \quad
|\bm p|= m_1+2 m_2+ 3 m_3 + \cdots .
\end{equation}
Then for a multi-partition $\cc_0 = (\rho^{(1)}, \rho^{(2)}, \dots , \rho^{(n)} ) \in \mathcal P(n)$, we set
\[
\N(\cc_0)= \N(\rho^{(1)} ) + \N(\rho^{(2)} ) + \cdots + \N(\rho^{(n)} ).
\]

Let $\mathscr C= \mathscr C_> \times \mathcal P(n) \times \mathscr C_<$.
We denote by $\CB$ the Kashiwara-Lusztig's canonical basis for $U_v^+(\mathfrak g)$, the positive part of the quantum affine algebra.

\begin{thm}[\cite{BCP, BeckNa}]\label{thm-BeNa}
There is a bijection $\eta \colon \CB \longrightarrow \mathscr C$ such that for each $\cc = (\cc_+, \cc_0, \cc_-) \in \mathscr C$, there exists a unique $b \in \CB$ satisfying
\begin{equation}\label{eqn-ec-final}
b \equiv E_{\cc_+} S_{\cc_{0} } E_{\cc_-} \bmod v^{-1} \ZZ[v^{-1}].
\end{equation}
\end{thm}

Now the number $\N(\cc)$ is defined by $\N(\cc)=\N(\cc_+)+\N(\cc_0)+\N(\cc_-)$ for each $\cc \in \mathscr C$. Using the canonical basis $\CB$, H.\ Kim and K.-H.\ Lee expanded the product side of the Gindikin-Karpelevich formula as a sum, and obtained the following theorem.

\begin{thm}[\cite{KL:11b}]\label{thm-KL}
We have
\begin{equation} \label{eqn-main-2}
\prod_{\alpha \in \Delta^+} \left ( \frac {1- q^{-1} \zz^\alpha}{1 - \zz^\alpha} \right )^{\mult(\alpha)} =  \sum_{b \in \CB} (1 - q^{-1})^{\N(\eta(b)) } \zz^{\wt(b)}.
\end{equation}
\end{thm}

\medskip

In the rest of this section, we will prove a combinatorial description of the formula \eqref{eqn-main-2} using the set $\YY(\infty)$ of reduced proper generalized Young walls.

We define a map $\theta\colon \mathcal P(n) \longrightarrow \KK$ as follows.  For $\cc_0=(\rho^{(1)}, \rho^{(2)}, \dots , \rho^{(n)} ) \in \mathcal P(n)$, we define
\[
\theta(\cc_0)= \sum_{i=1}^n m_{1,i} (\delta_i) + m_{2,i}(2\delta_i) +\cdots + m_{r,i}(r\delta_i) +\cdots ,
\]
where $\rho^{(i)} = (1^{m_{1,i}}2^{m_{2,i}} \cdots r^{m_{r,i}} \cdots)$ for $i=1, 2, \ldots , n$.  Then we define a map $\Theta\colon \mathscr C \longrightarrow \KK$ by
\[
\Theta(\cc) = \theta (\cc_0) + \sum_{i \in \ZZ} c_i (\beta_i) ,
\]
where $\cc=(\cc_+, \cc_0, \cc_-)$,  $\cc_+=(c_0, c_{-1}, c_{-2}, \dots)$, $\cc_-=(c_1, c_2, \dots)$ and $\beta_i$ is given by \eqref{beta} with $(\beta_i) \in \mathcal K$.

\begin{cor} \label{cor-k}
The map $\Theta\colon \mathscr C \longrightarrow \KK$ is a bijection, and for $\cc \in \mathscr C$, the number of distinct parts in $\pp = \Theta(\cc)$ is the same as $\N(\cc)$; i.e., $\N(\Theta(\cc))= \N(\cc)$.
\end{cor}

\begin{proof}
By Theorem \ref{thm-BeNa}, the set $\mathscr C$ parametrizes a PBW type basis of $U_v^+(\mathfrak g)$. Thus the set $\mathscr C$ also parametrizes a PBW basis of the universal enveloping algebra $U^+(\mathfrak g)$. Now the first assertion follows from the fact that the Kostant partitions parametrize the elements in a PBW basis of $U^+(\mathfrak g)$ and that the function $\Theta$ is defined according to these correspondences. The second assertion follows from the definitions of $\N$ for $\mathscr C$ and $\KK$, respectively.
\end{proof}

\begin{thm}\label{main}
Let $\mathfrak{g}$ be an affine Kac-Moody algebra of type $A_n^{(1)}$.  Then
\begin{equation} \label{eqn-m}
\prod_{\alpha \in \Delta^+} \left ( \frac{1-q^{-1}\zz^{\alpha}}{1-\zz^{\alpha}} \right )^{\mult(\alpha)} = \sum_{Y \in \YY(\infty)} (1-q^{-1})^{\N(Y)}\zz^{-\wt(Y)},
\end{equation}
where $\N(Y)$ is defined in \eqref{d-def}.
\end{thm}

\begin{proof}
By Lemma \ref{lem-unf}, Proposition \ref{prop-bi}, Theorem \ref{thm-BeNa} and Corollary \ref{cor-k}, we have bijections
\[
\CB \stackrel{\eta}{\longrightarrow}
\mathscr C \stackrel{\Theta}{\longrightarrow}
\KK \stackrel{\psi}{\longrightarrow}
\KK(\infty) \stackrel{\Phi}{\longrightarrow}
\YY(\infty).
\]
For $b \in \CB$, we write $Y=(\Phi \circ \psi \circ \Theta \circ \eta)(b) \in \YY(\infty)$.  Then, by Proposition \ref{prop-count} and Corollary \ref{cor-k}, we have $\N(\eta(b))=\N(Y)$. We also see from the constructions that $\wt(b)=-\wt (Y)$.
Now the equality \eqref{eqn-m} follows from Theorem \ref{thm-KL}.
\end{proof}

\section{Connection to Braverman-Finkelberg-Kazhdan's formula}

We briefly recall the framework of the paper \cite{BFK}. Let $G$ (resp.\ $\widehat G$) be the minimal (resp.\ formal) Kac-Moody group functor attached to a symmetrizable Kac-Moody root datum and let $\mathfrak g$ be the corresponding Lie algebra. There is a natural imbedding $G \hookrightarrow \widehat G$. The group $G$ has the closed subgroup functors $U \subset B$, $U_- \subset B_-$ such that the quotients $B/U$ and $B_-/U_-$ are naturally isomorphic to the Cartan subgroup $H$ of $G$. We denote by
$\widehat B$ and $\widehat U$ the closures of $B$ and $U$ in $\widehat G$, respectively. We will denote the coroot lattice of $G$ by $\Lambda$ and the set of positive coroots by $R^+ \subset \Lambda$. The subsemigroup of $\Lambda$ generated by $R^+$ will be denoted by $\Lambda^+$. For an element $\gamma=\sum a_i \alpha_i^\vee \in \Lambda^+$ with simple coroots $\alpha_i^\vee$, we write $|\gamma|= \sum a_i$.
We assume that $G$ is ``simply connected"; i.e., the lattice $\Lambda$ is equal to the cocharacter lattice of $H$.

We set $\mathcal F = \mathbf F_q((t))$ and $\mathcal O = \mathbf F_q[[t]]$, where $\mathbf F_q$ is the finite field with $q$ elements. We let $\mathrm{Gr}=\widehat G (\mathcal F) / \widehat G (\mathcal O)$. Each $\lambda \in \Lambda$ defines a homomorphism $\mathcal F^* \longrightarrow H(\mathcal F)$. We will denote the image of $t$ under this homomorphism by
$t^\lambda$, and its image in $\mathrm{Gr}$ will also be denoted by $t^\lambda$. We set
\[ S^\lambda = \widehat U(\mathcal F) \cdot t^\lambda \subset \mathrm{Gr} \quad \text{ and } \quad T^\lambda = U_-(\mathcal F) \cdot t^\lambda \subset \mathrm{Gr}. \]

In a recent paper \cite{BFK}, Braverman, Finkelberg and Kazhdan defined the generating function
\begin{equation} \label{eqn-Ig}
I_{\mathfrak g} (q)= \sum_{\gamma \in \Lambda^+} \#(T^{-\gamma} \cap S^0) q^{-|\gamma|} \zz^{\gamma}, \end{equation} and computed this sum as a product using a geometric method. Now we assume that the set of positive coroots $R^+$ forms a root system of type $A_n^{(1)}$, and we identify $R^+$ with the set of positive roots $\Delta^+$ in the previous sections of this paper. In this case, the resulting product in \cite{BFK} is
\[ I_{\mathfrak g} (q)= \prod_{i=1}^n \prod_{j=1}^\infty \frac{1-q^{-i}\zz^{j\delta}}{1-q^{-(i+1)}\zz^{j\delta}}
\prod_{\alpha\in\Delta^+} \left ( \frac{1-q^{-1}\zz^\alpha}{1-\zz^\alpha} \right )^{\mult(\alpha)}. \] We separate the factor
\[ \prod_{i=1}^n \prod_{j=1}^\infty \frac{1-q^{-i}\zz^{j\delta}}{1-q^{-(i+1)}\zz^{j\delta}}, \] and call it the {\em correction factor}. Our goal of this section is to write this correction factor and the function $I_{\mathfrak g}(q)$ as sums over reduced proper generalized Young walls.

\medskip

Let $\YY_0$ denote the subset of $\YY(\infty)$ consisting of the reduced proper generalized Young walls with empty rows in positions $\equiv 0 \bmod n+1$.  We define a map $\xi\colon \mathcal P(n) \longrightarrow \YY_0$ by the following assignment.  If $\bm p=(\rho^{(1)}, \dots , \rho^{(n)})$ is a multi-partition, then the parts of the partition $\rho^{(j)}$ give the lengths of the rows $\equiv j \bmod n+1$ in a reduced proper generalized Young wall $\xi(\bm p) = Y\in \YY_0$.  The following is clear from the definition.

\begin{lemma} \label{lem-xi}
The map $\xi\colon \PP(n) \longrightarrow \YY_0$ defined above is a bijection.
\end{lemma}

\begin{ex}\label{multiY}
Let $n=2$.  If
\[
\bm p = \left(\, \yng(5,2,1)\,,\, \yng(3,2,2)\, \right),
\]
then the corresponding element in $\YY_0$ is
\[
Y=\xi(\bm p) = \wall{{0,2,1,0,2},{1,0,2},{},{0,2},{1,0},{},{0}, {1,0}} \ .
\]
\end{ex}

\medskip

For $Y \in \YY_0$, define 
\[
\M(Y) = \sum_{i=1}^{n} (i+1) M_i(Y),
\]
where $M_i(Y)$ is the number of nonempty rows $\equiv i \bmod n+1$ in $Y$.  Moreover, we define $|Y|$ to be the total number of blocks in $Y$.

\begin{ex}
Let $Y$ be as in Example \ref{multiY}.  Then
\begin{gather*}
\M(Y) = 2\cdot3 + 3\cdot 3 = 15 \ \text{ and }\  |Y| = 15.
\end{gather*}
\end{ex}

Let us consider $\N(Y)$ for $Y \in \YY_0$, where $\N(Y)$ is defined in \eqref{d-def}. Since $Y$ has empty rows in positions $\equiv 0$ mod $n+1$, we have $(p_m, q_m)=(0,0)$ for all $m \ge 0$, and obtain $\mathscr Q(Y)=\{ 0 \}$ and $\mathscr P(Y) = 0$. Hence we have 
\begin{equation} \label{eqn-nn} \N(Y)= \sum_{j=1}^n \#(S_j(Y) \setminus \{ 0 \})  \qquad \text{ for } Y \in \YY_0 .
\end{equation}

\begin{prop}\label{correction}
Let $\mathfrak g$ be an affine Kac-Moody algebra of type $A_n^{(1)}$.  Then
\[
\prod_{i=1}^n \prod_{j=1}^\infty \frac{1-q^{-i}\zz^{j\delta}}{1-q^{-(i+1)}\zz^{j\delta}} = \sum_{Y\in \YY_0} (1-q)^{\N(Y)} q^{-\M(Y)}\zz^{|Y|\delta}.
\]
\end{prop}

\begin{proof}
We have 
\begin{align*} 
\prod_{j=1}^\infty \frac{1-q^{-i}\zz^{j\delta}}{1-q^{-(i+1)}\zz^{j\delta}} 
&=  \prod_{j=1}^\infty \left ( 1 + \sum_{k=1}^\infty (1 - q) q^{-k(i+1)} \zz^{kj\delta} \right ) \\
&= \sum_{\rho^{(i)} \in \PP(1)} (1-q)^{\N(\rho^{(i)})} q^{-(i+1)M(\rho^{(i)})} \zz^{|\rho^{(i)}|\delta}, 
\end{align*}
where $\N(\rho^{(i)}) = \#\{ r : m_r \neq 0 \}$ and $|\rho^{(i)}|=m_1+2m_2+ \cdots$ are defined in \eqref{eqn-q} and we set $M(\rho^{(i)})= m_1+ m_2+ \cdots$ for $\rho^{(i)}=(1^{m_1}2^{m_2} \cdots) \in \PP(1)$. For a multi-partition $\rho=(\rho^{(1)}, \dots , \rho^{(n)}) \in \PP(n)$, define \[\N(\rho)= \sum_{i=1}^n \N(\rho^{(i)}), \quad 
|\rho|= \sum_{i=1}^n |\rho^{(i)}| \quad \text{ and } \quad \M(\rho)= \sum_{i=1}^n (i+1)M(\rho^{(i)}) . \] 
Then we have 
\begin{align} 
\prod_{i=1}^n \prod_{j=1}^\infty \frac{1-q^{-i}\zz^{j\delta}}{1-q^{-(i+1)}\zz^{j\delta}} 
&= \prod_{i=1}^n \sum_{\rho^{(i)} \in \PP(1)} (1-q)^{\N(\rho^{(i)})} q^{-(i+1)M(\rho^{(i)})} \zz^{|\rho^{(i)}|\delta} \nonumber \\ 
&= \sum_{\rho \in \PP(n)} (1-q)^{\N(\rho)} q^{-\M(\rho)} \zz^{|\rho|\delta} . \label{ww}
\end{align}
Using the map $\xi$ in Lemma \ref{lem-xi}, one can see that $\N(\rho)=\N(\xi(\rho))$, $\M(\rho)=\M(\xi(\rho))$ and $|\rho|= |\xi(\rho)|$ for $\rho \in \PP(n)$, and the proposition follows from \eqref{ww}.
\end{proof}

The following formula  provides a combinatorial description of the affine Gindikin-Karpelevich formula proved by Braverman, Finkelberg and Kazhdan.

\begin{cor} \label{cor-end}
When ${\mathfrak g}$ is an affine Kac-Moody algebra of type
$A_{n}^{(1)}$, we have
\begin{multline} \label{mul}
I_{\mathfrak g}(q)=\prod_{i=1}^n \prod_{j=1}^\infty \frac{1-q^{-i}\zz^{j\delta}}{1-q^{-(i+1)}\zz^{j\delta}}
\prod_{\alpha\in\Delta^+} \left ( \frac{1-q^{-1}\zz^\alpha}{1-\zz^\alpha} \right )^{\mult(\alpha)} \\
=
\sum_{(Y_1,Y_2) \in \YY(\infty)\times\YY_0}
(1-q^{-1})^{\N(Y_1)}(1-q)^{\N(Y_2)}q^{-\M(Y_2)}\zz^{-\wt(Y_1)+|Y_2|\delta}.
\end{multline}
\end{cor}

Furthermore, comparing \eqref{mul} with \eqref{eqn-Ig}, we obtain a combinatorial formula for the number of points in the intersection $T^{-\gamma} \cap S^0$:
\begin{cor} \label{cor-nup}
We have
\[ 
\#(T^{-\gamma} \cap S^0) = \sum_{\substack{ (Y_1,Y_2) \in \YY(\infty)\times\YY_0 \\ -\wt(Y_1)+|Y_2|\delta = \gamma} }(1-q^{-1})^{\N(Y_1)}(1-q)^{\N(Y_2)} q^{|\gamma |-\M(Y_2)} ,
\]
where $\gamma \in \Lambda^+$ is identified with the corresponding element of the root lattice of $\mathfrak g$.
\end{cor}

\begin{ex}
Assume $n=1$ and $\gamma=\delta$. Then we have 
\[ (Y_1, Y_2) = \left 
(\varnothing \ ,
	\begin{tikzpicture}[baseline,scale=.45]
		\GYW{{0}}
	\end{tikzpicture} \ 
\right ), \quad  
\left (
	\begin{tikzpicture}[baseline,scale=.45]
		\GYW{{0,1}}
	\end{tikzpicture}
\ , \varnothing \right ), 
\quad \text{or} \quad 
\left (
	\begin{tikzpicture}[baseline=10,scale=.45]
		\GYW{{},{1,0}}
	\end{tikzpicture}
\ , \varnothing \right ). 
 \]
From the first pair, we get $(1-q^{-1})^0(1-q)^1 q^{2-2}=1-q$. The second yields $(1-q^{-1})^1(1-q)^0q^{2-0}=q^2-q$, and the third $(1-q^{-1})^2 (1-q)^0 q^{2-0} = (q-1)^2$. Thus we have
\[ \#(T^{-\gamma} \cap S^0) = 1-q + q^2-q + (q-1)^2 = 2(q-1)^2 .\]
\end{ex}

\appendix
\section{Implementation in Sage}

Together with Lucas Roesler and Travis Scrimshaw, the fourth named author has implemented generalized Young walls and the statistics developed here in the open-source mathematical software Sage \cite{combinat,sage}.  We conclude with some examples using our package.

First we may verify examples given above.  To verify Example \ref{ex-nosplit}, we have the following, where \lstinline!Y.number_of_parts()! refers to $\N(Y)$. 

\begin{lstlisting}
sage: Yinf = InfinityCrystalOfGeneralizedYoungWalls(2)
sage: Y = Yinf([[0,2,1,0,2],[1,0,2],[2],[],[1]])
sage: Y.pp()
         1|
          |
         2|
     2|0|1|
 2|0|1|2|0|
sage: Y.number_of_parts()
4
\end{lstlisting}

Similarly, to see Examples \ref{ex-split} and \ref{ex-splitalg} using Sage, use the following commands.

\begin{lstlisting}
sage: Yinf = InfinityCrystalOfGeneralizedYoungWalls(2)
sage: row1 = [0,2,1]
sage: row2 = [1,0]
sage: row3 = [2,1,0,2,1,0,2,1,0]
sage: row6 = [2,1,0,2,1,0]
sage: Y = Yinf([row1,row2,row3,[],[],row6])
sage: Y.pp()
       0|1|2|0|1|2|
                  |
                  |
 0|1|2|0|1|2|0|1|2|
               0|1|
             1|2|0|
sage: Y.number_of_parts()
8
sage: Y.pp()
       0|1|2|0|1|2|
                  |
                  |
 0|1|2|0|1|2|0|1|2|
               0|1|
             1|2|0|
sage: Y.number_of_parts()
8
\end{lstlisting}

Note that the remaining crystal structure pertaining to generalized Young walls has also been implemented.  We continue using the $Y$ from the previous example.

\begin{lstlisting}
sage: Y.weight()
-7*alpha[0] - 7*alpha[1] - 6*alpha[2]
sage: Y.f(1).pp()
       0|1|2|0|1|2|
                 1|
                  |
 0|1|2|0|1|2|0|1|2|
               0|1|
             1|2|0|
sage: Y.e(0).pp()
         1|2|0|1|2|
                  |
                  |
 0|1|2|0|1|2|0|1|2|
               0|1|
             1|2|0|
sage: Y.content()
20
\end{lstlisting}

One may also generate the top part of the crystal graph.  

\begin{lstlisting}
sage: Yinf = InfinityCrystalOfGeneralizedYoungWalls(2)
sage: S = Yinf.subcrystal(max_depth=4)
sage: G = Yinf.digraph(subset=S)
sage: view(G,tightpage=True)
\end{lstlisting}

We conclude by mentioning that highest weight crystals realized by generalized Young walls have also been implemented in Sage, following Theorem 4.1 of \cite{KS:10}.

\begin{lstlisting}
sage: Delta = RootSystem(['A',3,1])
sage: P = Delta.weight_lattice()
sage: La = P.fundamental_weights()
sage: YLa = CrystalOfGeneralizedYoungWalls(3,La[0])
sage: S = YLa.subcrystal(max_depth=6)
sage: G = YLa.digraph(subset=S)
sage: view(G,tightpage=True)
\end{lstlisting}

\bibliography{GK-A-affine}{}

\def\cprime{$'$}
\providecommand{\bysame}{\leavevmode\hbox to3em{\hrulefill}\thinspace}
\begin{thebibliography}{10}

\bibitem{BCP}
J.~Beck, V.~Chari, and A.~Pressley, \emph{An algebraic characterization of the
  affine canonical basis}, Duke Math. J. \textbf{99} (1999), no.~3, 455--487.

\bibitem{BeckNa}
J.~Beck and H.~Nakajima, \emph{Crystal bases and two-sided cells of quantum
  affine algebras}, Duke Math. J. \textbf{123} (2004), no.~2, 335--402.

\bibitem{BFK}
A.~Braverman, M.~Finkelberg, and D.~Kazhdan, \emph{Affine
  {G}indikin-{K}arpelevich formula via {U}hlenbeck spaces}, Contributions in
  analytic and algebraic number theory, Springer Proc. Math., vol.~9, Springer,
  New York, 2012, pp.~17--29.

\bibitem{BBF:11}
B.~Brubaker, D.~Bump, and S.~Friedberg, \emph{Weyl group multiple {D}irichlet
  series: type {A} combinatorial theory}, Annals of Mathematics Studies, vol.
  175, Princeton University Press, Princeton, NJ, 2011.

\bibitem{BN:10}
D.~Bump and M.~Nakasuji, \emph{Integration on {$p$}-adic groups and crystal
  bases}, Proc. Amer. Math. Soc. \textbf{138} (2010), no.~5, 1595--1605.

\bibitem{Ga:04}
H.~Garland, \emph{Certain {E}isenstein series on loop groups: convergence and
  the constant term}, Algebraic groups and arithmetic, Tata Inst. Fund. Res.,
  Mumbai, 2004, pp.~275--319.

\bibitem{GK:62}
S.~G. Gindikin and F.~I. Karpelevi{\v{c}}, \emph{Plancherel measure for
  symmetric {R}iemannian spaces of non-positive curvature}, Dokl. Akad. Nauk
  SSSR \textbf{145} (1962), 252--255.

\bibitem{HK:02}
J.~Hong and S.-J. Kang, \emph{Introduction to quantum groups and crystal
  bases}, Graduate Studies in Mathematics, vol.~42, American Mathematical
  Society, Providence, RI, 2002.

\bibitem{HL:08}
J.~Hong and H.~Lee, \emph{Young tableaux and crystal {$\mathcal{B}(\infty)$}
  for finite simple {L}ie algebras}, J. Algebra \textbf{320} (2008), no.~10,
  3680--3693.

\bibitem{Kam:10}
J.~Kamnitzer, \emph{Mirkovi\'c-{V}ilonen cycles and polytopes}, Ann. of Math.
  (2) \textbf{171} (2010), no.~1, 245--294.

\bibitem{Kang:03}
S.-J. Kang, \emph{Crystal bases for quantum affine algebras and combinatorics
  of {Y}oung walls}, Proc. London Math. Soc. (3) \textbf{86} (2003), no.~1,
  29--69.

\bibitem{Kash:91}
M.~Kashiwara, \emph{On crystal bases of the {$Q$}-analogue of universal
  enveloping algebras}, Duke Math. J. \textbf{63} (1991), no.~2, 465--516.

\bibitem{Kash:95}
\bysame, \emph{On crystal bases}, Representations of groups ({B}anff, {AB},
  1994), CMS Conf. Proc., vol.~16, Amer. Math. Soc., Providence, RI, 1995,
  pp.~155--197.

\bibitem{Kash:02}
\bysame, \emph{Bases cristallines des groupes quantiques}, Cours
  Sp\'ecialis\'es [Specialized Courses], vol.~9, Soci\'et\'e Math\'ematique de
  France, Paris, 2002, Edited by Charles Cochet.

\bibitem{KN:94}
M.~Kashiwara and T.~Nakashima, \emph{Crystal graphs for representations of the
  {$q$}-analogue of classical {L}ie algebras}, J. Algebra \textbf{165} (1994),
  no.~2, 295--345.

\bibitem{KS:97}
M.~Kashiwara and Y.~Saito, \emph{Geometric construction of crystal bases}, Duke
  Math. J. \textbf{89} (1997), no.~1, 9--36.

\bibitem{KL:11}
H.~H. Kim and K.-H. Lee, \emph{Representation theory of {$p$}-adic groups and
  canonical bases}, Adv. Math. \textbf{227} (2011), no.~2, 945--961.

\bibitem{KL:11b}
\bysame, \emph{Quantum affine algebras, canonical bases, and {$q$}-deformation
  of arithmetical functions}, Pacific J. Math. \textbf{255} (2012), no.~2,
  393--415.

\bibitem{KS:10}
J.-A. Kim and D.-U. Shin, \emph{Generalized {Y}oung walls and crystal bases for
  quantum affine algebra of type {$A$}}, Proc. Amer. Math. Soc. \textbf{138}
  (2010), no.~11, 3877--3889.

\bibitem{Lan:71}
R.~P. Langlands, \emph{Euler products}, Yale University Press, New Haven,
  Conn., 1971, A James K. Whittemore Lecture in Mathematics given at Yale
  University, 1967, Yale Mathematical Monographs, 1.

\bibitem{Lee:07}
H.~Lee, \emph{Realizations of crystal {$\mathcal{B}(\infty)$} using {Y}oung
  tableaux and {Y}oung walls}, J. Algebra \textbf{308} (2007), no.~2, 780--799.

\bibitem{LS-A}
K.-H. Lee and B.~Salisbury, \emph{A combinatorial description of the
  {G}indikin-{K}arpelevich formula in type {$A$}}, J. Combin. Theory Ser. A
  \textbf{119} (2012), no.~5, 1081--1094.

\bibitem{Lit:95}
P.~Littelmann, \emph{Paths and root operators in representation theory}, Ann.
  of Math. (2) \textbf{142} (1995), no.~3, 499--525.

\bibitem{Luszt:90}
G.~Lusztig, \emph{Canonical bases arising from quantized enveloping algebras},
  J. Amer. Math. Soc. \textbf{3} (1990), no.~2, 447--498.

\bibitem{Luszt:91}
\bysame, \emph{Quivers, perverse sheaves, and quantized enveloping algebras},
  J. Amer. Math. Soc. \textbf{4} (1991), no.~2, 365--421.

\bibitem{Luszt:93}
\bysame, \emph{Introduction to quantum groups}, Progress in Mathematics, vol.
  110, Birkh\"auser Boston Inc., Boston, MA, 1993.

\bibitem{Mac:71}
I.~G. Macdonald, \emph{Spherical functions on a group of {$p$}-adic type},
  Ramanujan Institute, Centre for Advanced Study in Mathematics,University of
  Madras, Madras, 1971, Publications of the Ramanujan Institute, No. 2.

\bibitem{McN:11}
P.~J. McNamara, \emph{Metaplectic {W}hittaker functions and crystal bases},
  Duke Math. J. \textbf{156} (2011), no.~1, 1--31.

\bibitem{combinat}
The {S}age-{C}ombinat community, \emph{{S}age-{C}ombinat: enhancing {S}age as a
  toolbox for computer exploration in algebraic combinatorics}, 2008, {\tt
  http://combinat.sagemath.org}.

\bibitem{sage}
W.\thinspace{}A. Stein et~al., \emph{{S}age {M}athematics {S}oftware ({V}ersion
  5.12)}, The Sage Development Team, 2013, {\tt http://www.sagemath.org}.

\end{thebibliography}
\bibliographystyle{amsplain}
\end{document}